\documentclass[a4paper,11pt]{amsart}

\usepackage[utf8]{inputenc}
\usepackage{amssymb}
\usepackage[T1]{fontenc}
\usepackage{mathptmx}

\makeatletter
\@namedef{subjclassname@2020}{%
  \textup{2020} Mathematics Subject Classification}
\makeatother

\newtheorem{theorem}{Theorem}[section]

\newtheorem{lemma}[theorem]{Lemma}
\newtheorem{corollary}[theorem]{Corollary}
\theoremstyle{definition}

\numberwithin{equation}{section}

\begin{document}

\title[Maps preserving two-sided zero products]{Maps preserving two-sided zero products on Banach algebras}

\author{M. Bre\v sar}

\address{M. Bre\v sar, Faculty of Mathematics and Physics,  University of Ljubljana,
 and Faculty of Natural Sciences and Mathematics, University 
of Maribor, Slovenia}
 \email{matej.bresar@fmf.uni-lj.si}

\author{M. L. C. Godoy} \author{A. R. Villena}
\address{M. L. C. Godoy and A.R. Villena, Departamento de An\' alisis
Matem\' atico, Fa\-cul\-tad de Ciencias, Universidad de Granada,
 Granada, Spain} 
 \email{mgodoy@ugr.es}
 \email{avillena@ugr.es}

\thanks{The first author was supported by the Slovenian Research Agency (ARRS) Grant P1-0288. The second and the third author were supported by MCIU/AEI/FEDER Grant PGC2018-093794-B-I00 and Junta de Andaluc\' ia Grant FQM-185. The second author was also supported by MIU Grant FPU18/00419 and MIU Grant EST19/00466. The third author was also supported by Proyectos I+D+i del programa operativo FEDER-Andaluc\'ia Grant A-FQM-484-UGR18.}



\begin{abstract}
Let $A$ and $B$ be Banach algebras 
with  bounded approximate identities and let $\Phi:A\to B$ be a surjective continuous linear map 
which preserves two-sided zero products (i.e., $\Phi(a)\Phi(b)=\Phi(b)\Phi(a)=0$ whenever $ab=ba=0$). 
We show that $\Phi$ is a weighted Jordan homomorphism provided that $A$ is zero product determined 
and weakly amenable.  
These conditions are in particular fulfilled when $A$ is the group algebra $L^1(G)$ with $G$ any 
locally compact group. We also study a more general type of continuous linear maps $\Phi:A\to B$ that satisfy 
$\Phi(a)\Phi(b)+\Phi(b)\Phi(a)=0$ whenever $ab=ba=0$. We show in  particular  that if $\Phi$ is surjective and $A$ is a $C^*$-algebra, then $\Phi$ is a weighted Jordan homomorphism.
\end{abstract}

\maketitle

\section{Introduction}\label{s1}

Let $A$ and $B$ be Banach algebras. 
We will say that a linear map $\Phi:A\to B$ {\em preserves two-sided zero products} if for all $a,b\in A$,
\begin{equation}\label{00}
ab=ba=0 \ \implies \Phi(a)\Phi(b)= \Phi(b)\Phi(a)=0.
\end{equation}
Obvious examples of such maps are homomorphisms and antihomomorphisms. 
Their common generalizations are 
 {\em Jordan homomorphisms}, i.e., linear maps $\Psi:A\to B$ satisfying
$$\Psi(a\circ b)=\Psi(a)\circ \Psi(b)\quad\forall a,b\in A,$$
where $a\circ b$ stands for the Jordan product $ab+ba$.
Under the  mild assumption that the centre of $B$ 
 does not contain nonzero nilpotent elements, every Jordan homomorphism from $A$ onto $B$
also preserves  two-sided zero products \cite[Lemma 7.20]{zpdbook}. 
Next we recall that a linear map $W:B\to B$ is called a {\em centralizer} if 
$$W(ab)=W(a)b = aW(b)\quad\forall a,b\in B.$$
We say that $\Phi$ is a {\em weighted Jordan homomorphism} if there exist an invertible centralizer 
$W$ of $B$ and a Jordan homomorphism $\Psi$ from $A$ to $B$ such that $\Phi=W\Psi$. Observe that 
$\Phi$ preserves two-sided zero products if and only if $\Psi$ does. We also remark that, by the 
closed graph theorem,  every centralizer $W$ is automatically continuous if $B$ is a faithful algebra 
(i.e., $bB=Bb=\{0\}$ implies $b=0$), and so, in this case, $\Phi$ is continuous if and only if $\Psi$ is.

Is every surjective continuous linear map $\Phi:A\to B$ which preserves two-sided zero products a weighted Jordan homomorphism?
This question is similar to but, as it turns out, more difficult
than a more thoroughly studied question of describing zero products preserving 
continuous linear maps 
(see the most recent publications \cite{zpdbook, G, GP, K, MS}  for historical remarks and references).
% which preserves zero products  is a weighted homomorphism (see \cite[Section 7.1]{zpdbook}).
 It is known that the answer is positive if either $A$ 
and $B$ are $C^*$-algebras \cite[Theorem 3.3]{ABEV1} or if $A=L^1(G)$ and $B=L^1(H)$ where $G$ and $H$ are locally 
compact groups with $G\in {\rm [SIN]}$  (i.e.,  $G$ has a
base of compact neighborhoods of the identity that is invariant under all inner automorphisms) \cite[Theorem 3.1\,(i)]{AEV}.
In fact,  \cite[Theorem 3.3]{ABEV1} does not require that $\Phi$ satisfies \eqref{00} but only that
for all $a,b\in A$, 
\begin{equation}\label{11}
ab=ba=0 \implies \ \Phi(a)\circ \Phi(b)=0.
\end{equation}
This condition was also considered in the recent  algebraic paper  \cite{BG}. Observe also  that it is more general 
than the condition that $\Phi$ preserves zero Jordan products ($a\circ b=0$ implies $\Phi(a)\circ \Phi(b)=0$) studied 
 in \cite{C}.

The  goal of this paper is to generalize and unify  the aforementioned results from \cite{ABEV1} and \cite{AEV}.
Our approach is based on the concept of a  {\em zero product determined} Banach algebra. These are Banach algebras 
$A$ with the property that every continuous bilinear functional $\varphi:A\times A\to \mathbb{C}$ satisfying 
$\varphi(a,b)=0$ whenever $ab=0$ is of the form $\varphi(a,b)=\tau(ab)$ for some continuous linear functional 
$\tau$ on $A$. We refer to the recent book \cite{zpdbook} for a survey of these algebras.
Let us for now only mention   that they form a fairly large class of Banach algebras whose main representatives, 
$C^*$-algebras and group algebras of locally compact groups, are also weakly amenable Banach algebras having 
bounded approximate identities. 

In Section \ref{s2}, we show that the answer to our question is positive, i.e.,  
a surjective continuous linear map $\Phi:A\to B$ which preserves two-sided zero products is a weighted Jordan homomorphism, 
provided that $A$ is zero product determined and 
weakly amenable, and, additionally, both $A$ and $B$ have  bounded approximate identities (Theorem \ref{24}).  
This in particular implies that that the restriction in \cite[Theorem 3.1\,(i)]{AEV} that  $G\in {\rm [SIN]}$ 
is redundant  (Corollary \ref{25}).
  
Section \ref{s3} is devoted to condition \eqref{11}. We show that \cite[Theorem 3.3]{ABEV1} still holds if $B$ 
is any Banach algebra  with a bounded approximate identity, not only  a $C^*$-algebra (Theorem \ref{t32}). Our 
second main result regarding \eqref{11} considers the case where $A=\mathcal{A}(X)$ is the algebra of approximable 
operators (Theorem \ref{tapp}).

\section{Condition \eqref{00}}
%Addressing the case of group algebras}
\label{s2}

Throughout, for a Banach space $X$,
we write $X^*$ for the dual of $X$ and 
$\langle\,\cdot\,,\,\cdot\,\rangle$ for the duality between $X$ and $X^*$.
Let $A$ be a Banach algebra. We turn $A^*$ into a Banach $A$-bimodule by letting
\[
\langle b,a\cdot\omega\rangle=\langle ba,\omega\rangle,\quad
\langle b,\omega\cdot a\rangle=\langle ab,\omega\rangle\quad
\forall a,b\in A, \ \forall \omega\in A^*.
\]
The space of continuous derivations from $A$ into $A^*$ is denoted by $\mathcal{Z}^1(A,A^*)$.
The main representatives of $\mathcal{Z}^1(A,A^*)$ are the so-called inner derivations. 
The inner derivation implemented by  $\omega\in A^*$ is the map 
$\delta_\omega\colon A\to A^*$ defined by
\[
\delta_\omega(a)=a\cdot\omega-\omega\cdot a\quad\forall a\in A.
\]
The Banach algebra is called \emph{weakly amenable} if every element of $\mathcal{Z}^1(A,A^*)$ is inner.
For a thorough treatment of this property and an account of many interesting examples of weakly amenable 
Banach algebras we refer the reader to \cite{D}. We should remark that  
the group algebra $L^1(G)$ of each locally compact group $G$ and
each $C^*$-algebra are weakly amenable \cite[Theorems 5.6.48 and 5.6.77]{D}.

Our first result is a sharpening of \cite[Theorem 2.7]{ABEV2019} (see also \cite[Theorem 6.6]{zpdbook}).

\begin{theorem}\label{t1}
Let $A$ be a Banach algebra, and suppose that:
\begin{enumerate}
\item[(a)]
$A$ is zero product determined;
\item[(b)]
$A$ has a bounded approximate identity;
\item[(c)]
$A$ is weakly amenable.
\end{enumerate}
Then there exists a constant $C\in\mathbb{R}^+$ such that 
for each continuous bilinear functional $\varphi\colon A\times A\to \mathbb{C}$ with the property that for all $a,b\in A$,
\begin{equation}\label{1436}
 ab=ba=0 \ \implies \ \varphi(a,b)=0
\end{equation}
there exist $\sigma,\tau\in A^*$ such that
\[
\Vert \sigma\Vert\le C\Vert\varphi\Vert,\quad
\Vert\tau\Vert\le C\Vert\varphi\Vert,\]
and
\[
\varphi(a,b)=\sigma(ab)+\tau(ba)\quad\forall a,b\in A.
\]
\end{theorem}

\begin{proof}
Let $(e_\lambda)_{\lambda\in\Lambda}$ be an approximate identity for $A$ of bound $M$.

Since $A$ is weakly amenable, the map $\omega\mapsto\delta_\omega$
from $A^*$ to the Banach space $\mathcal{Z}^1(A,A^*)$ is a continuous linear surjection, 
and so, by the open mapping theorem, there exists a constant $N\in\mathbb{R}^+$ such that, 
for each $D\in\mathcal{Z}^1(A,A^*)$, there exists an $\omega\in A^*$ with
\[
\Vert\omega\Vert\le N\Vert D\Vert
\]
and
\[
a\cdot\omega-\omega\cdot a=D(a)\quad\forall a\in A.
\]

Towards the proof of the theorem,
we proceed through a detailed inspection of the proof of \cite[Theorem 2.7]{ABEV2019}.

Define $\varphi_1,\varphi_2\colon A\times A\to\mathbb C$ by
\begin{equation*}
\begin{split}
\varphi_1(a,b)
&= 
\tfrac{1}{2}\bigl[\varphi(a,b) + \varphi(b,a)\bigr],\\
\varphi_2(a,b)
&= 
\tfrac{1}{2}\bigl[\varphi(a,b) - \varphi(b,a)\bigr]
\quad
\forall a,b\in A.
\end{split}
\end{equation*}
It is clear that both $\varphi_1$ and $\varphi_2$ satisfy condition  \eqref{1436}
and that 
\[
\Vert\varphi_1\Vert\le\Vert\varphi\Vert,
\quad
\Vert\varphi_2\Vert\le\Vert\varphi\Vert.
\]

Since $\varphi_1$ is symmetric it follows from the last assertion of \cite[Lemma 2.6]{ABEV2019} (or \cite[Theorem 6.1]{zpdbook}) that
there exists a $\xi\in A^*$ such that
\begin{equation}\label{mm}
2\varphi_1(a,b)=\langle a\circ b,\xi\rangle
\quad
\forall a,b\in A.
\end{equation}
For each $a\in A$,
we observe that
\[
\langle a\circ e_\lambda,\xi\rangle=2\varphi_1(a,e_\lambda)
\quad\forall\lambda\in\Lambda
\]
and hence that
\[
\vert
\langle a\circ e_\lambda,\xi\rangle\vert
\le
2M\Vert\varphi\Vert\Vert a\Vert
\quad\forall\lambda\in\Lambda.
\]
Taking limit and using  $\lim_{\lambda\in\Lambda} a\circ e_\lambda=2a$ we see that
\[
2\langle a,\xi\rangle\le 2M\Vert\varphi\Vert\Vert a\Vert.
\]
This shows that
\begin{equation}\label{1437}
\Vert\xi\Vert\le M\Vert\varphi\Vert.
\end{equation}

Our next concern will be the behaviour of the skew-symmetric functional $\varphi_2$.
By \cite[Lemma 4.1]{ABEV2017} (or \cite[Theorem 6.1]{zpdbook}),
there exists a $\psi\in A^*$ such that
\[
\varphi_2(ab,c)-\varphi_2(b,ca)+\varphi_2(bc,a)=\langle abc,\psi\rangle
\quad
\forall a,b,c\in A.
\]
The proof  reveals that 
the functional $\psi$ is defined by
\[
\langle a,\psi\rangle=\lim_{\lambda\in\Lambda}\varphi_2(a,e_\lambda)
\quad\forall a\in A,
\]
so that
\begin{equation}\label{1438}
\Vert\psi\Vert\le M\Vert\varphi\Vert.
\end{equation}

By \cite[Lemma 2.6]{ABEV2019} (or the proof of \cite[Theorem 6.5]{zpdbook}) , the map
\[
D\colon A\to A^*,
\quad
\langle b,D(a)\rangle=\varphi_2(a,b)+\frac{1}{2}\langle a\circ b,\psi\rangle
\]
is a continuous derivation, and clearly (using \eqref{1438})
\[
\Vert D\Vert\le \Vert\varphi_2\Vert+\Vert\psi\Vert\le (1+M)\Vert\varphi\Vert.
\]
Consequently, there exists an $\omega\in A^*$ such that
\begin{equation}\label{1439}
\Vert\omega\Vert\le N\Vert D\Vert\le N(1+M)\Vert\varphi\Vert
\end{equation}
and
\[
a\cdot \omega - \omega\cdot a=D(a) 
\quad
\forall a\in A,
\]
and hence
\[
\langle ba-ab,\omega\rangle-  \varphi_2(a,b) = \tfrac{1}{2}\langle a\circ b,\psi\rangle
\quad 
\forall a,b\in A.
\]
Viewing this expression as a bilinear functional on $A\times A$,  
we see that the left-hand side is skew-symmetric and the right-hand side is symmetric. 
Therefore, both sides are zero.
Thus 
\begin{equation}\label{mm2}
\varphi_2(a,b)= \langle ba-ab,\omega\rangle
\quad 
\forall a,b\in A.
\end{equation}
We then define
\[
\sigma=\tfrac{1}{2}\xi-\omega,
\quad
\tau=\tfrac{1}{2}\xi+\omega.
\]
From 
\eqref{1437} and \eqref{1439} we see that
\[
\Vert\sigma\Vert\le(\tfrac{1}{2}M+N+NM)\Vert\varphi\Vert,
\]
\[
\Vert\tau\Vert\le(\tfrac{1}{2}M+N+NM)\Vert\varphi\Vert,
\]
and from \eqref{mm} and \eqref{mm2} we deduce that
\[
\varphi(a,b)=\sigma(ab)+\tau(ba)
\quad
\forall a,b\in A.\qedhere 
\]
\end{proof}

We can now start our consideration of maps preserving two-sided zero products.

\begin{lemma}\label{t2}
Let $A$ be a Banach algebra, and suppose  that:
\begin{enumerate}
\item[(a)]
$A$ is zero product determined;
\item[(b)]
$A$ has a bounded approximate identity;
\item[(c)]
$A$ is weakly amenable.
\end{enumerate}
Let $B$ be a Banach algebra and
let $\Phi\colon A\to B$ be a continuous linear map which preserves two-sided zero products.
Then there exist:
\begin{itemize}
\item
a closed left ideal $L$ of $B$ containing $\Phi(A)$ and 
a continuous linear map $U\colon L\to B$,
\item
a closed right ideal $R$ of $B$ containing $\Phi(A)$ and
a continuous linear map $V\colon R\to B$
\end{itemize}
such that
\[
U(xy)=xU(y),
\quad
V(zx)=V(z)x
\quad
\forall x\in B, \ \forall y\in L, \ \forall z\in R,
\]
\[
U(\Phi(a))=V(\Phi(a))
\quad\forall a\in A,
\]
and
\[
U(\Phi(a\circ b))=V(\Phi(a\circ b))=\Phi(a)\circ\Phi(b)
\quad
\forall a,b\in A.
\]
\end{lemma}

\begin{proof}
Let $(e_\lambda)_{\lambda\in\Lambda}$ be an approximate identity for $A$ of bound $M$, and
let $C$ be the constant given in Theorem \ref{t1}.

We define  
\[
L=\left
\{x\in B\colon \text{ the net } \bigl(x\Phi(e_\lambda)\bigr)_{\lambda\in\Lambda}\text{ is convergent}\right\},
\]
\[
U\colon L\to B,\quad
U(x)=\lim_{\lambda\in\Lambda}x\Phi(e_\lambda) \quad\forall x\in L,
\]
and
\[
R=\left
\{x\in B\colon \text{ the net } \bigl(\Phi(e_\lambda)x\bigr)_{\lambda\in\Lambda}\text{ is convergent}\right\},
\]
\[
V\colon R\to B,\quad
V(x)=\lim_{\lambda\in\Lambda}\Phi(e_\lambda)x \quad\forall x\in R.
\]
It is clear that $L$ is a left ideal of $B$, $R$ is a right ideal of $B$, and
routine verifications,
using that the net $(\Phi(e_\lambda))_{\lambda\in\Lambda}$ is bounded, 
show that both $L$ and $R$ are closed subspaces of $B$.
It is also obvious that both $U$ and $V$ are continuous linear maps with
$\Vert U\Vert\le M\Vert\Phi\Vert$ and $\Vert V\Vert\le M\Vert\Phi\Vert$ and that
\[
U(xy)=xU(y),
\quad
V(zx)=V(z)x
\quad
\forall x\in B, \ \forall y\in L, \ \forall z\in R.
\]

We claim that 
\begin{equation}\label{1111}
\bigl(\Phi(a^2)\Phi(e_\lambda)\bigr)_{\lambda\in\Lambda}\to\Phi(a)^2\quad\forall a\in A.
\end{equation}
Fix an $a\in A$. By the Hahn-Banach theorem, for each $\lambda\in\Lambda$ there exists a $\xi_\lambda\in B^*$
with $\Vert\xi_\lambda\Vert=1$ and 
\begin{equation}\label{1112}
\bigl\langle \Phi(a^2)\Phi(e_\lambda)-\Phi(a)^2,\xi_\lambda\bigr\rangle=
\bigl\Vert \Phi(a^2)\Phi(e_\lambda)-\Phi(a)^2\bigr\Vert.
\end{equation}
For each $\lambda\in\Lambda$, we consider the continuous bilinear functional
\[
\varphi_\lambda\colon A\times A\to\mathbb{C},
\quad
\varphi_\lambda(u,v)=\langle \Phi(u)\Phi(v),\xi_\lambda\rangle
\quad
\forall u,v\in A,
\]
which clearly satisfies \eqref{1436} and 
\[
\Vert\varphi_\lambda\Vert\le\Vert\Phi\Vert^2.
\]
Hence Theorem \ref{t1} yields the existence of $\sigma_\lambda,\tau_\lambda\in A^*$ such that
\[
\Vert\sigma_\lambda\Vert
\le 
C\Vert\Phi\Vert^2,
\quad
\Vert\tau_\lambda\Vert
\le 
C\Vert\Phi\Vert^2,
\]
and
\begin{equation}\label{1113}
\langle\Phi(u)\Phi(v),\xi_\lambda\rangle=
\sigma_\lambda(uv)+\tau_\lambda(vu)
\quad
\forall u,v\in A.
\end{equation}
From \eqref{1112} and \eqref{1113} we deduce that
\begin{equation}\label{1114}
\begin{split}
\Vert\Phi(a^2)\Phi(e_\lambda)-\Phi(a)^2\Vert
&=
\langle \Phi(a^2)\Phi(e_\lambda),\xi_\lambda\rangle
-\langle\Phi(a)^2,\xi_\lambda\rangle\\
&=
\sigma_\lambda(a^2e_\lambda)+\tau_\lambda(e_\lambda a^2)
-\sigma_\lambda(a^2)-\tau_\lambda(a^2)\\
&=
\sigma_\lambda(a^2e_\lambda-a^2)+
\tau_\lambda(e_\lambda a^2-a^2) 
\quad
\forall \lambda\in\Lambda.
\end{split}
\end{equation}
We now observe that
\begin{equation*}
\begin{split}
\vert \sigma_\lambda(a^2e_\lambda-a^2)\vert
&\le
C\Vert\Phi\Vert^2\Vert a^2e_\lambda-a^2\Vert,\\
\vert\tau_\lambda(e_\lambda a^2-a^2)\vert
&\le
C\Vert\Phi\Vert^2\Vert e_\lambda a^2-a^2\Vert
\quad\forall \lambda\in\Lambda,
\end{split}
\end{equation*}
and so, taking limits and using that
\[
\lim_{\lambda\in\Lambda}\Vert a^2e_\lambda-a^2\Vert=
\lim_{\lambda\in\Lambda}\Vert e_\lambda a^2-a^2\Vert=0,
\]
we see that
\[
\lim_{\lambda\in\Lambda} \sigma_\lambda(a^2e_\lambda-a^2)=
\lim_{\lambda\in\Lambda} \tau_\lambda(e_\lambda a^2-a^2)=0.
\]
Taking limit in \eqref{1114}  we now deduce that
\[
\lim_{\lambda\in\Lambda}\Vert\Phi(a^2)\Phi(e_\lambda)-\Phi(a)^2\Vert=0,
\]
which gives \eqref{1111}.

Of course, \eqref{1111} gives
\begin{equation}\label{117a}
\Phi(a^2)\in L,  \quad U(\Phi(a^2))=\Phi(a)^2 \quad\forall a\in A.
\end{equation}

In the same way as \eqref{1111} one proves that
\begin{equation*}
\bigl(\Phi(e_\lambda)\Phi(a^2)\bigr)_{\lambda\in\Lambda}\to\Phi(a)^2\quad\forall a\in A,
\end{equation*}
which clearly yields
\begin{equation}\label{117b}
\Phi(a^2)\in R, \quad V(\Phi(a^2))=\Phi(a)^2 \quad\forall a\in A.
\end{equation}

It remains to prove that $\Phi(A)\subset L\cap R$.
From \eqref{117a} and \eqref{117b} we deduce immediately that
\begin{equation}\label{1341}
\begin{split}
&\Phi(a\circ b)\in L\cap R,\\
%\quad
U(\Phi(a\circ b))=&V(\Phi(a\circ b))=\Phi(a)\circ\Phi(b)\quad\forall a,b\in A.\end{split}
\end{equation}
For each $a\in A$, 
\cite[Theorem II.16]{AL} gives $b,c\in A$ such that $a=bcb$,
so that 
\[
a=\tfrac{1}{2}b\circ (b\circ c)-\tfrac{1}{2}b^2\circ c
\]
and \eqref{1341} then gives $\Phi(a)\in L\cap R$ and further
$U(\Phi(a))=V(\Phi(a))$.
\end{proof}

\begin{lemma}\label{t3}
Let $A$ and $B$ be Banach algebras, and suppose  that:
\begin{enumerate}
\item[(a)]
$A$ is zero product determined;
\item[(b)]
$A$ has a bounded approximate identity;
\item[(c)]
$A$ is weakly amenable;
\item[(d)]
$B$ is faithful.
\end{enumerate}
Let $\Phi\colon A\to B$ be a continuous linear map having dense range and preserving two-sided zero products.
Then there exists an injective continuous centralizer $W\colon B\to B$ such that
$$W(\Phi(a\circ b))=\Phi(a)\circ\Phi(b)\quad\forall a,b\in A.$$
\end{lemma}

\begin{proof}
We apply Lemma \ref{t2}.
Since $\Phi$ has dense range, it follows that $L=R=B$ and that $U=V$.
Set $W=U\,(=V)$. Then $W$ is a centralizer on $B$ and
\[
W(\Phi(a\circ b))=\Phi(a)\circ\Phi(b)\quad\forall a,b\in A.
\]
The only point remaining concerns the injectivity of $W$.
We claim that
\begin{equation}\label{1732}
\ker W B^3=B^3\ker W=\{0\}.
\end{equation}
Let $x\in\ker W$.
For each $a\in A$, we have
\[
0=W(x)\Phi(a^2)=W(x\Phi(a^2))=xW(\Phi(a^2))=x\Phi(a)^2,
\]
and, since the range of $\Phi$ is dense, we arrive at
\[
xy^2=0\quad\forall y\in B.
\]
We thus get
\[
x(yz+zy)=0\quad\forall x\in\ker W, \ \forall y,z\in B.
\]
For all $x\in\ker W$ and $y,z,w\in B$ we have (using that $xz\in\ker W$)
\[
(xyz)w=
(-xzy)w=-(xz)yw=
(xz)wy=x(zw)y=
-xy(zw),
\]
whence $xyzw=0$, and so $\ker W B^3=\{0\}$.
Similarly we see that $B^3\ker W=\{0\}$. Thus, \eqref{1732} holds.

It is an elementary exercise to show that an element $b$ in a faithful algebra $B$ satisfying  $b B^3=B^3b =\{0\}$ must be $0$. Indeed, one first observes that every $c\in B^2 b B^2$ satisfies $cB=Bc=\{0\}$, which yields $B^2 b B^2 = \{0\}$. Similarly we see that this implies $B^2 b B =  BbB^2 =\{0\}$, hence $BbB=B^2b=bB^2=\{0\}$, and finally $bB=Bb=\{0\}$. Therefore, $b=0$.

Thus, \eqref{1732} shows that $\ker W=\{0\}$.
\end{proof}

\begin{lemma}\label{l4}
Let $A$ and $B$ be Banach algebras, and 
suppose that $B$ has a bounded approximate identity.
Let $\Phi\colon A\to B$ be a surjective linear map, and
let $W\colon B\to B$ be a linear map such that
$$W(\Phi(a\circ b))=\Phi(a)\circ\Phi(b)\quad\forall
a,b\in A.$$ Then $W$ is surjective.
\end{lemma}

\begin{proof}
Set $x\in B$.
By \cite[Theorem II.16]{AL}, there exist $y,z\in B$ such that $x=yzy$,
so that 
\[
x=\tfrac{1}{2}y\circ (y\circ z)-\tfrac{1}{2}y^2\circ z.
\]
Since $\Phi$ is surjective, we can choose $a,b,c,d\in A$ with
\[
\Phi(a)=y,\quad
\Phi(b)=y\circ z,\quad
\Phi(c)=y^2,\quad
\Phi(d)=z.
\]
The condition on $W$ now gives
\begin{align*}
%\begin{split}
W\bigl(
\Phi(\tfrac{1}{2}a\circ b-\tfrac{1}{2}c\circ d)\bigr)
&=
\tfrac{1}{2}\Phi(a)\circ\Phi(b)-
\tfrac{1}{2}\Phi(c)\circ\Phi(d)\\
&=
\tfrac{1}{2}y\circ(y\circ z)-\tfrac{1}{2}y^2\circ z
=
x.
%\end{split}
\qedhere 
\end{align*}
\end{proof}

We are now ready to establish our main result.

\begin{theorem}\label{24}
Let $A$ and $B$ be Banach algebras, and  suppose that:
\begin{enumerate}
\item[(a)]
$A$ is zero product determined;
\item[(b)]
$A$ has a bounded approximate identity;
\item[(c)]
$A$ is weakly amenable;
\item[(d)]
$B$ has a bounded approximate identity.
\end{enumerate}
Let $\Phi\colon A\to B$ be a surjective continuous linear map which preserves two-sided zero products.
Then $\Phi$ is a weighted Jordan homomorphism.
\end{theorem}

\begin{proof}
Since $B$ has a bounded approximate identity, it follows that $B$ is faithful.
We conclude from Lemma \ref{t3} that there exists an injective continuous centralizer $W\colon B\to B$ such that
\begin{equation}\label{1807}
W(\Phi(a\circ b))=\Phi(a)\circ\Phi(b)\quad\forall a,b\in A.
\end{equation}
Lemma \ref{l4} now shows that $W$ is surjective.

Having proved that $W$ is an invertible centralizer, we can define $\Psi=W^{-1}\Phi$
which is a surjective continuous linear map and, further, we deduce from \eqref{1807}
that $\Psi$ is a Jordan homomorphism. Of course, $\Phi=W\Psi$.
\end{proof}

The crucial examples of zero product determined Banach algebras are
the group algebras $L^1(G)$ for each locally compact group $G$
and $C^*$-algebras \cite[Theorems 5.19 and 5.21]{zpdbook}.
Furthermore, these Banach algebras are also weakly amenable and have bounded approximate identities.
Therefore, it is legitimate to apply Theorem \ref{24} in the case where $A$ is a group algebra or a 
$C^*$-algebra.

\begin{corollary}\label{26}
Let $G$ be a locally compact group,
let $B$ be a Banach algebra having a bounded approximate identity, and
let $\Phi\colon L^1(G)\to B$ be a surjective continuous linear map which preserves two-sided zero products.
Then $\Phi$ is a weighted Jordan homomorphism.
\end{corollary}

Our final corollary generalizes  \cite[Theorem 3.1\,(i)]{AEV}.
\begin{corollary}\label{25}
Let $G$ and $H$ be locally compact groups, and 
let $\Phi\colon L^1(G)\to L^1(H)$ be a surjective continuous linear map which preserves two-sided zero products.
Then there exist  
a surjective continuous Jordan homomorphism $\Psi\colon L^1(G)\to L^1(H)$ and
an invertible central measure $\mu\in M(H)$ such that 
$\Phi(f)=\mu\ast\Psi(f)$ for each $f\in L^1(G)$.
\end{corollary}

\begin{proof}
By Corollary \ref{26}, there exist an invertible cetralizer $W$ of $L^1(H)$ and a 
surjective continuous Jordan homomorphism $\Psi\colon L^1(G)\to L^1(H)$ such that
$\Phi=W\Psi$.
The centralizer $W$ can be thought of as an element of the centre of the multiplier algebra
of $L^1(H)$ which is, by Wendel's Theorem (see \cite[Theorem 3.3.40]{D}), isomorphic to
the measure algebra $M(H)$. This gives a measure $\mu\in M(H)$ as required.
\end{proof}

\section{Condition \eqref{11}}
%Addressing the case of operator algebras}
\label{s3}

We will not discuss Theorem \ref{24} in the case where $A$ is a $C^*$-algebra,
because in this case condition \eqref{00} can be weakened to condition \eqref{11}.
Showing this is the main purpose of this section.

\begin{lemma}\label{31}
Let $A$ be a Banach algebra, and suppose  that:
\begin{enumerate}
\item[(a)]
$A$ is zero product determined;
\item[(b)]
$A$ has a bounded approximate identity.
\end{enumerate}
Let $B$ be a Banach algebra and
let $\Phi\colon A\to B$ be a continuous linear map 
satisfying condition \eqref{11}.
%such that for all $a,b\in A$,
%\begin{equation*}
%ab=ba=0 \implies \ \Phi(a)\circ\Phi(b)=0.
%\end{equation*}
Then there exist a closed linear subspace $J$ of $B$ containing $\Phi(A)$ and 
a continuous linear map $W\colon J\to B$ such that
$$W(\Phi(a\circ b))=\Phi(a)\circ\Phi(b)\quad\forall a,b\in A.$$
Moreover, if $B$ has a bounded approximate identity and  $\Phi$ is surjective, then  $W$ is a surjective map from $B$ onto itself.
\end{lemma}

\begin{proof}
Let $(e_\lambda)_{\lambda\in\Lambda}$ be an approximate identity for $A$ of bound $M$.

We define  
\[
J=\left
\{x\in B\colon \text{ the net } \bigl(\Phi(e_\lambda)\circ x\bigr)_{\lambda\in\Lambda}\text{ is convergent}\right\}
\]
and
\[
W\colon J\to B,\quad
W(x)=\lim_{\lambda\in\Lambda}\tfrac{1}{2} \Phi(e_\lambda)\circ x \quad\forall x\in J.
\]
It is clear that $J$ is a linear subspace of $B$ and routine verifications,
using that the net $(\Phi(e_\lambda))_{\lambda\in\Lambda}$ is bounded, 
show that $J$ is a closed linear subspace of $B$.
It is also obvious that $W$ is a continuous linear map with
$\Vert W\Vert\le M\Vert\Phi\Vert$.

Applying  \cite[Theorem 6.1 and Remark 6.2]{zpdbook}  to the continuous bilinear map $\varphi\colon A\times A\to B$ defined by
\[
\varphi(a,b)=\Phi(a)\circ\Phi(b)
\quad\forall a,b\in A
\]
we see  that there exists a continuous linear map $S\colon A\to B$ such that
\begin{equation}\label{1653}
\Phi(a)\circ\Phi(b)=S(a\circ b)
\quad\forall a,b\in A.
\end{equation}
%For this purpose
%we can apply \cite[Theorem 1.2]{ABEV1} (or \cite[Theorem 6.1, Remark 6.2]{zpdbook}) to the continuous bilinear map $\varphi\colon A\times A\to B$ defined by
%\[
%\varphi(a,b)=\Phi(a)\circ\Phi(b)
%\quad\forall a,b\in A.
%\]
%Actually, the result is stated in \cite{ABEV1} only for the case where $A$ is a $C^*$-algebra, but
%an examination of the proof immediately reveals that it holds for every zero product preserving Banach
%algebra having a bounded approximate identity (see also \cite[Lemma 4.1]{ABEV2017}).
For each $a\in A$, we thus have
\[
\Phi(e_\lambda)\circ\Phi(a)=S(e_\lambda\circ a)
\quad\forall\lambda\in\Lambda.
\]
Using  $\lim_{\lambda\in\Lambda}e_\lambda\circ a=2a$ and the continuity of $S$,
we see by taking limit that
\[
\lim_{\lambda\in\Lambda}\Phi(e_\lambda)\circ\Phi(a)= 2 S(a).
\]
This shows that $\Phi(a)\in J$ and that $W(\Phi(a))=S(a)$.
On the other hand, using \eqref{1653}, we see that
\begin{equation}\label{1654}
W(\Phi(a\circ b))=S(a\circ b)=\Phi(a)\circ\Phi(b)
\quad\forall a,b\in A.
\end{equation}

Of course, if $\Phi$ is surjective, then $J=B$. 
Now suppose that, in addition, $B$ has a bounded approximate identity.
Then, on account of \eqref{1654}, Lemma \ref{l4} shows that $W$ is surjective.
\end{proof}

\begin{lemma}\label{32}
Let $A$ and $B$ be a Banach algebras,  
let $\Phi\colon A\to B$ be a continuous linear map, and
let $\omega\in B$.
Suppose that:
\begin{enumerate}
\item[(a)]
$A$ is the closed linear span of its idempotents.
\item[(b)]
$\Phi(a^2)\circ\omega=2\Phi(a)^2$ for each $a\in A$.
%$W(\Phi(a^2))=\Phi(a)^2$ for each $a\in  A$.
%\item[(c)]
%There exists $\omega\in B$ such that $2W(\Phi(a))=\Phi(a)\circ \omega$ for each $a\in A$.
\end{enumerate}
Then $\omega^2\Phi(a)=\Phi(a)\omega^2$ for each $a\in A$.
\end{lemma}

\begin{proof}
Let $e\in A$ be an idempotent.
From (b) we see that
\begin{equation}\label{1722b}
\omega\Phi(e)+\Phi(e)\omega=\Phi(e)\circ\omega=\Phi(e^2)\circ\omega=2\Phi(e)^2.
\end{equation}
By multiplying \eqref{1722b} by $\Phi(e)$ on the left we obtain
\begin{equation}\label{1723b}
\Phi(e)\omega\Phi(e)+\Phi(e)^2\omega=2\Phi(e)^3
\end{equation}
and multiplying by $\Phi(e)$ on the right we get
\begin{equation}\label{1724b}
\omega\Phi(e)^2+\Phi(e)\omega\Phi(e)=2\Phi(e)^3.
\end{equation}
From \eqref{1723b} and \eqref{1724b} we arrive at
$\omega\Phi(e)^2= \Phi(e)^2\omega$,
which, on account of \eqref{1722b}, yields
\[
\omega^2\Phi(e)=\Phi(e)\omega^2.
\]
Since $A$ is the closed linear span of its idempotents, it follows that
\[
\omega^2 \Phi(a)=\Phi(a)\omega^2
\quad\forall a\in A.\qedhere 
\]
\end{proof}

In the proof of the next results we will use the first Arens product on the second dual $A^{**}$ 
of a Banach algebra $A$. We will  denote this product by juxtaposition.
Furthermore,
we will use the following basic facts about the weak* continuity of the first Arens product
which the reader can find in \cite{D}.
\begin{enumerate}
\item[(A1)]
For each $a\in A$, the map $\xi\mapsto a\xi$ from $A^{**}$ to itself is weak* continuous.
\item[(A2)]
For each $\xi\in A^{**}$, the map $\zeta\mapsto\zeta\xi$ from $A^{**}$ to itself is weak* continuous.
\item[(A3)]
If $A$ is a $C^*$-algebra, then the product in $A^{**}$ is separately weak* continuous.
\end{enumerate}

\begin{theorem}\label{t32}
Let $A$ be  a $C^*$-algebra, 
let $B$ be a Banach algebra having a bounded approximate identity, and
let $\Phi\colon A\to B$ be a surjective continuous linear map such that for all $a,b\in A$,
\begin{equation*}
ab=ba=0 \implies \ \Phi(a)\circ\Phi(b)=0.
\end{equation*}
Then $\Phi$ is a weighted Jordan homomorphism.
\end{theorem}

\begin{proof}
By Lemma \ref{31} there exists a surjective continuous linear map $W\colon B\to B$ such that
\begin{equation}\label{1340}
W(\Phi(a\circ b))=\Phi(a)\circ\Phi(b)\quad\forall a,b\in A.
\end{equation}

We write 
$\Phi^{**}\colon A^{**}\to B^{**}$ and $W^{**}\colon B^{**}\to B^{**}$ 
for the second duals of the continuous linear maps 
$\Phi\colon A\to B$ and $W\colon B\to B$, respectively.
We claim that
\begin{equation}\label{2005}
W^{**}(\Phi^{**}(x\circ y))=\Phi^{**}(x)\circ\Phi^{**}(y)
\quad\forall x,y\in A^{**}.
\end{equation}
Set $x,y\in A^{**}$, and take nets $(a_i)_{i\in I}$ and $(b_j)_{j\in J}$ in $A$ such that
\begin{equation*}
\begin{split}
(a_i)_{i\in I}\to x,&\\
(b_j)_{j\in J}\to y&\quad\text{in }(A^{**},\sigma(A^{**},A^*)).
\end{split}
\end{equation*}
On account of \eqref{1340}, we have
\begin{equation}\label{1340b}
W(\Phi(a_ib_j+b_ja_i))=\Phi(a_i)\Phi(b_j)+\Phi(b_j)\Phi(a_i)\quad\forall i\in I, \ \forall j\in J,
\end{equation}
and the task is now to take the iterated limit $\lim_{j\in J}\lim_{i\in I}$ on both sides of the above equation.
Throughout the proof, the limits $\lim_{i\in I}$ and $\lim_{j\in J}$ are taken with respect to the weak* topology.
From (A3)
and the weak* continuity of both $\Phi^{**}$ and $W^{**}$ we deduce that
\begin{equation}\label{1205}
\begin{split}
\lim_{j\in J}\lim_{i\in I}W^{**}\bigl(\Phi^{**} (a_ib_j+b_ja_i)\bigr)
&=
\lim_{j\in J}W^{**}\bigl(\Phi^{**}(xb_j+b_jx)\bigr)\\
&=
W^{**}\bigl(\Phi^{**}(xy+yx)\bigr).
\end{split}
\end{equation}
From (A1)-(A2) (applied to the Arens product of $B^{**}$) and the weak* continuity of $\Phi^{**}$
we deduce that
\begin{equation}\label{1206}
\begin{split}
\lim_{j\in J}\lim_{i\in I}\Phi(b_j)\Phi(a_i)
&=
\lim_{j\in J}\Phi(b_j)\Phi^{**}(x)\\
&=
\Phi^{**}(y)\Phi^{**}(x).
\end{split}
\end{equation}
The remaining iterated limits
\[
\lim_{j\in J}\lim_{i\in I}\Phi(a_i)\Phi(b_j)
\]
must be treated with much more care than the previous ones.
We regard $A$ as a $C^*$-algebra acting on the Hilbert space of its universal representation, 
and we regard the continuous bilinear map 
\[
A\times A\to B,\quad
(a,b)\to\Phi(a)\Phi(b)
\]
as a continuous bilinear map with values in the Banach space $B^{**}$ which 
is separately ultraweak-weak* continuous.
By applying \cite[Theorem 2.3]{JKR}, 
we obtain that the bilinear map above extends uniquely, without change of norm, to a continuous bilinear map 
$\phi\colon A^{**}\times A^{**}\to B^{**}$ which is separately weak* continuous.
From this, and using (A1)-(A2) and the weak* continuity of $\Phi^{**}$, we obtain
\begin{equation}\label{1207}
\begin{split}
\lim_{j\in J}\lim_{i\in I}\Phi(a_i)\Phi(b_j)
&=
\lim_{j\in J}\lim_{i\in I}\phi(a_i,b_j)\\
&=
\phi(x,y)\\
&=
\lim_{i\in I}\lim_{j\in J}\phi(a_i,b_j)\\
&=
\lim_{i\in I}\lim_{j\in J}\Phi(a_i)\Phi(b_j)\\
&=
\lim_{i\in I}\Phi(a_i)\Phi^{**}(y)\\
&=
\Phi^{**}(x)\Phi^{**}(y).
\end{split}
\end{equation}
From \eqref{1340b}, \eqref{1205}, \eqref{1206}, and \eqref{1207}, it may be concluded that
\begin{align*}
W^{**}(\Phi^{**}(x\circ y))=&
\lim_{j\in J}\lim_{i\in I}W(\Phi(a_i\circ b_j))
\\
=&\lim_{j\in J}\lim_{i\in I}\Phi(a_i)\Phi(b_j)+
\lim_{j\in J}\lim_{i\in I}\Phi(b_j)\Phi(a_i)\\=&
\Phi^{**}(x)\circ\Phi^{**}(y),
\end{align*}
and \eqref{2005} is proved.

Define $\omega=\Phi^{**}(1)\in B^{**}$, where $1$ is the unit of the von Neumann algebra $A^{**}$.
Setting $x=y$ in \eqref{2005} we conclude that
\begin{equation}\label{1459a}
W^{**}(\Phi^{**}(x^2))=\Phi^{**}(x)^2
\quad\forall x\in A^{**},
\end{equation}
and setting $y=1$ in \eqref{2005} we see that
\begin{equation}\label{1459}
2W^{**}(\Phi^{**}(x))=\omega\Phi^{**}(x)+\Phi^{**}(x)\omega
\quad\forall x\in A^{**}.
\end{equation}
From \eqref{1459a} and \eqref{1459} we deduce that
\[
\Phi^{**}(x^2)\circ\omega=2\Phi^{**}(x)^2\quad\forall x\in A^{**}.
\]
Since $A^{**}$ is a von Neumann algebra, it is the closed linear span of its projections
and we are in a position to apply Lemma \ref{32}, which gives
\begin{equation*}
\omega^2 \Phi^{**}(x))=\Phi^{**}(x)\omega^2\quad\forall x\in A^{**}.
\end{equation*}
In particular,
\begin{equation*}
\omega^2 \Phi(a)=\Phi(a)\omega^2
\quad\forall a\in A.
\end{equation*}
Since $\Phi$ is surjective, it may be concluded that
\begin{equation*}
\omega^2 u=u\omega^2\quad\forall u\in B.
\end{equation*}
From \eqref{1459} we see that, for each $a\in A$,
\begin{align*}
\omega W(\Phi(a))=&
\omega\tfrac{1}{2}\bigl(\omega\Phi(a)+\Phi(a)\omega\bigr)\\=&
\tfrac{1}{2}\bigl(\omega^2\Phi(a)+\omega\Phi(a)\omega\bigr)
\\
=&\tfrac{1}{2}\bigl(\Phi(a)\omega^2+\omega\Phi(a)\omega\bigr)\\=&
\tfrac{1}{2}\bigl(\Phi(a)\omega+\omega\Phi(a)\bigr)\omega\\=&
W(\Phi(a))\omega.
\end{align*}
Since both $\Phi$ and $W$ are surjective, it may be concluded that
\begin{equation}\label{1130c}
\omega u=u\omega\quad\forall u\in B.
\end{equation}
From \eqref{1459} we now deduce that 
\[
W(\Phi(a))=
\tfrac{1}{2}\bigl(\omega\Phi(a)+\Phi(a)\omega\bigr)=\omega\Phi(a)
\quad\forall a\in A,
\]
and hence that
\[
W(u)=\omega u\quad\forall u\in B.
\]
Furthermore, for all $a,b\in B$, using \eqref{1130c} we obtain
\[
W(ab)=\omega ab=W(a)b,\]\[
W(ab)=(\omega a)b=(a\omega)b=aW(b),
\]
whence $W$ is a centralizer on $B$.
In order to prove that $W$ is an invertible centralizer, it remains to show that $W$ is injective.
If $a\in\ker W$, then
\[
aB=aW(B)=W(a)B=\{0\}
\]
and therefore $a=0$.

Since $W$ is an invertible centralizer on $B$, \eqref{1340} shows that $W^{-1}\Phi$ is a Jordan
homomorphism, and hence  $\Phi$ is a weighted Jordan homomorphism. 
\end{proof}

Our final concern will be the algebra $\mathcal{A}(X)$  of approximable operators on a Banach space $X$.
It is shown in \cite{ABEV0} that $\mathcal{A}(X)$ has the so-called property $\mathbb{B}$ for each Banach
space $X$ (see also \cite[Example 5.15]{zpdbook}).
Further, it is known that $\mathcal{A}(X)$ has a bounded left approximate identity if and only if the Banach
space $X$ has the bounded approximation property (see \cite[Theorem 2.9.37]{D}). In this case,
$\mathcal{A}(X)$ is actually a zero product determined Banach algebra (see \cite[Lemma 2.3]{ABEV0} or, alternatively,
\cite[Proposition 5.5]{zpdbook}). Another remarkable feature of $\mathcal{A}(X)$ is that it has
a bounded approximate identity if and only if $X^*$ has the bounded approximation property (see \cite[Theorem 2.9.37]{D}).

\begin{theorem}\label{t33}\label{tapp}
Let $X$ be a Banach space such that $X^*$ has the bounded approximation property,
let $B$ be a Banach algebra having a bounded approximate identity, and
let $\Phi\colon \mathcal{A}(X)\to B$ be a surjective continuous linear map such that for all $S,T\in \mathcal{A}(X)$,
\begin{equation*}
ST=TS=0 \implies \ \Phi(S)\circ\Phi(T)=0.
\end{equation*}
Then $\Phi$ is a weighted Jordan homomorphism.
\end{theorem}

\begin{proof}
We begin by applying Lemma \ref{31} to obtain a surjective continuous linear map $W\colon B\to B$ such that
\begin{equation}\label{1719}
W(\Phi(S\circ T))=\Phi(S)\circ\Phi(T)
\quad
\forall S,T\in \mathcal{A}(X).
\end{equation}

Let $(E_\lambda)_{\lambda\in\Lambda}$ be a bounded approximate identity for $\mathcal{A}(X)$.
Then we regard $\bigl(\Phi(E_\lambda)\bigr)_{\lambda\in\Lambda}$ as a bounded net in the second dual $B^{**}$
of $B$. It follows from the Banach-Alaoglu theorem that this net has a $\sigma(B^{**},B^*)$-convergent subnet.
Hence, by passing to a subnet, we may supose that $(E_\lambda)_{\lambda\in\Lambda}$ is a bounded approximate 
identity for $\mathcal{A}(X)$ such that 
\[
\lim_{\lambda\in\Lambda}\Phi(E_\lambda)= \omega
\quad\text{in }(B^{**},\sigma(B^{**},B^*))
\]
for some $\omega\in B^{**}$.

Set $T\in\mathcal{A}(X)$. 
Writting $E_\lambda$ for $S$ in \eqref{1719}, we obtain 
\begin{equation}\label{1720}
W(\Phi(E_\lambda T+TE_\lambda))=\Phi(E_\lambda)\Phi(T)+\Phi(T)\Phi(E_\lambda)\quad\forall \lambda\in\Lambda,
\end{equation}
and our next goal is to take limits on both sides of \eqref{1720}. 
Since
\[
\lim_{\lambda\in\Lambda}(E_\lambda T+TE_\lambda)= 2T
\quad\text{in }(\mathcal{A}(X),\Vert\cdot\Vert), 
\]
the continuity of $W\Phi$ gives
\begin{equation}\label{114}
\lim_{\lambda\in\Lambda}W(\Phi(E_\lambda T+TE_\lambda))=2W(\Phi(T))
\quad\text{in } (B,\Vert\cdot\Vert).
\end{equation}
On the other hand,
since
\[
\lim_{\lambda\in\Lambda}\Phi(E_\lambda)=\omega
\quad\text{in }(B^{**},\sigma(B^{**},B^*))
\]
and $\Phi(T)\in B$, we can appeal to (A1)-(A2) to deduce that
\begin{equation}\label{115}
\begin{split}
\lim_{\lambda\in\Lambda}
\Phi(E_\lambda)\Phi(T)=\omega\Phi(T)
&,\\
\lim_{\lambda\in\Lambda}
\Phi(T)\Phi(E_\lambda)=\Phi(T)\omega
&\quad\text{in }(B^{**},\sigma(B^{**},B^*)).
\end{split}
\end{equation}
Hence, taking limits in \eqref{1720} and using \eqref{114} and \eqref{115},
we obtain
\begin{equation}\label{1721}
2W(\Phi(T))=\omega\Phi(T)+\Phi(T)\omega.
\end{equation}

Having  \eqref{1719} and \eqref{1721} and
 using that  $\mathcal{A}(X)$ is the closed linear span of its idempotents, 
we can now apply Lemma \ref{32} to obtain
\[
\omega^2 \Phi(T)=\Phi(T)\omega^2
\quad\forall T\in\mathcal{A}(X).
\]
From the surjectivity of $\Phi$  we deduce that
\begin{equation*}
\omega^2 a=a\omega^2
\quad\forall a\in B.
\end{equation*} 
By using the same method as in the proof of Theorem \ref{t32} we verify that
$\omega a=a\omega$ for each $a\in B$, that $W(a)=\omega a$ for each $a\in A$,
and that $W$ is an invertible centralizer on $B$ such that $W^{-1}\Phi$ is a Jordan
homomorphism. 
\end{proof}


\begin{thebibliography}{99}

\bibitem{AL}
M. Akkar, M. Laayouni.
Th\'eor\`emes de factorisation dans les alg\`ebres completes de Jordan.
\emph{Collect. Math.} \textbf{46} (1995), 239--254.

\bibitem{ABEV0}
J. Alaminos, M. Bre\v sar, J. Extremera, A. R. Villena,
Maps preserving zero products,
\emph{Studia Math.} \textbf{193} (2009), 131--159.

\bibitem{ABEV1}
J. Alaminos, M. Bre\v sar, J. Extremera, A.\,R. Villena,
 Characterizing Jordan maps on $C ^*$-algebras through zero products, 
{\em Proc. Edinb. Math. Soc.} {\bf 53} (2010),  543--555.
 
\bibitem{ABEV2017}
J. Alaminos, M. Bre\v{s}ar, J. Extremera, A. R. Villena, 
Zero Lie product determined Banach algebras, 
\emph{Studia Math.} \textbf{239} (2017) 189--199.

\bibitem{ABEV2019}
J. Alaminos, M. Bre\v{s}ar, J. Extremera, A. R.Villena,
Zero Lie product determined Banach algebras, II,
\emph{J. Math. Anal. Appl.} \textbf{474} (2019), 1498--1511.

\bibitem{AEV} J. Alaminos, J. Extremera, A.\,R. Villena,
Orthogonality preserving linear maps on group algebras,
{\em Math. Proc. Cambridge Philos. Soc.} {\bf 158} (2015),  493--504.

\bibitem{zpdbook}
M.\ Bre\v sar, 
{\em Zero Product Determined Algebras}, 
Frontiers in Mathematics,  Birk\-h\" aus\-er/Springer, 2021. 

\bibitem{BG} M. Bre\v sar,  M. L. C. Godoy, Weighted Jordan homomorphisms, preprint, arXiv:
2111.15232.

\bibitem{C}M.~A.~
Chebotar, W.-F. Ke,  P.-H. Lee, R. Zhang, 
On maps preserving zero Jordan products, {\em 
Monatsh. Math. } {\bf 149} (2006), 91--101.

\bibitem{D}
H. G. Dales,
\emph{Banach algebras and automatic continuity.}
London Mathematical Society Monographs, New Series, 24, Oxford Science Publications,
The Clarendon Press, Oxford University Press, New York, 2000.

\bibitem{G}J. J. 
 Garc\' es, 
Complete orthogonality preservers between $C^*$-algebras, {\em
J. Math. Anal. Appl.} {\bf  483}  (2020),  123596, 16 pp.

\bibitem{GP}J. J. 
 Garc\' es,  A. M. Peralta, One-parameter groups of orthogonality preservers on $C^*$-algebras, {\em  Banach J. Math. Anal.} {\bf  15}  (2021),  24, 23 pp.

\bibitem{JKR}
B. E. Johnson, R. V. Kadison, J. R. Ringrose, 
Cohomology of operator algebras. III: reduction to normal cohomology.
\emph{Bull. Soc. Math. France} \textbf{100} (1972), 73--96.

\bibitem{K} T.
Kochanek,
Approximately order zero maps between $C^*$-algebras, {\em 
J. Funct. Anal.} {\bf  281} (2021),  109025, 49 pp.

\bibitem{MS}S. Y.
Manjegani, J  Soltani Farsani, 
On the Fourier algebra of certain hypergroups, {\em
Quaest. Math.} {\bf  43} (2020),  1513--1525.

\end{thebibliography}
\end{document}